\documentclass{gLMA2e}

\usepackage{amssymb,amsmath,graphicx,mathrsfs,amsbsy}

\usepackage{color}
\usepackage{amscd,verbatim}
\usepackage{graphicx}
\usepackage{wrapfig}

\newcommand{\bone}{\ensuremath{\bm{1}}}
\def\cG{\mathcal G}
\newcommand{\R}[1]{{\rm I\! R}^{#1}}

\usepackage{enumerate}
\usepackage{bm}

\everymath{\displaystyle}

\usepackage{epstopdf}
\usepackage{subfigure}

\theoremstyle{plain}
\newtheorem{theorem}{Theorem}[section]

\newtheorem{lemma}[theorem]{Lemma}

\theoremstyle{definition}

\theoremstyle{remark}

\begin{document}



\title{On the Maximal Error of Spectral Approximation of Graph Bisection}

\author{
\name{John C. Urschel\textsuperscript{a,b,c,$\ast$}\thanks{$^\ast$Corresponding author. Email: jcurschel@gmail.com}
\and 
Ludmil T. Zikatanov\textsuperscript{c,d,$\dagger$}\thanks{$\dagger$The work of this author was supported in part by NSF DMS-1418843 and NSF DMS-1522615.}
}
\affil{\textsuperscript{a}Baltimore Ravens, NFL, Owings Mills, MD, USA;
  \textsuperscript{b}Department of Mathematics, Massachusetts Institute of Technology,
  Cambridge, MA, USA.
  \textsuperscript{c}Department of Mathematics, Penn State University,
  University Park, PA, USA;
  \textsuperscript{d}Institute for Mathematics and 
  Informatics, Bulgarian Academy of Sciences, Sofia, Bulgaria
}  
}
\maketitle

\begin{abstract}
  Spectral graph bisections are a popular heuristic
    aimed at approximating the solution of the NP-complete graph
  bisection problem. This technique, however, does
    not always provide a robust tool for graph partitioning. Using a
  special class of graphs, we prove that the
  standard spectral graph bisection can produce bisections that are
  far from optimal. In particular, we show that the maximum error in
  the spectral approximation of the optimal bisection (partition sizes
  exactly equal) cut for such graphs is bounded
  below by a constant multiple of the order of the graph squared.
\end{abstract}


\begin{classcode}05C40, 05C50 \end{classcode}

\section{Introduction and preliminaries}

A graph bisection is a partition of the vertex set into two parts of
equal order, thereby creating two subgraphs.  A bisection is
considered good if the number of edges between the two partitions is
small. Finding the graph bisection that minimizes
the number of edges between the two partitions is NP-complete
\cite{garey1979computers}. Despite this, graph bisections have found
application in scientific computing, VLSI design, and task scheduling
\cite{Metis0,Metis1}. A variety of heuristic algorithms have been
implemented in an attempt to approximate the optimal graph
bisection. One of the most popular techniques
  approximates the optimal cut via the zero level set of the
  discrete Laplacian eigenvector corresponding to
  the smallest non-zero eigenvalue. Naturally, such technique is
  called a spectral bisection. The efficient construction and the
  properties of spectral bisections is and has been an active area of
research
\cite{1990PothenA_SimonH_LiouK-aa,1988PowersD-aa,urschel2014cascadic}. Experiments
have shown that this technique works well in practice, and it has been
proven that spectral bisection works well on bounded degree planar
graphs and finite element meshes \cite{spielman1996spectral}.
For general graphs, however, this is not the case,
  and, in what follows, we show that spectral bisection is not a
  robust technique for approximating the optimal cut. In particular,
we construct a special class of graphs for which the
maximum error in the spectral approximation of the optimal bisection
is bounded below by a constant multiple of the order squared.

We begin the technical discussion by briefly
  introducing notation and definitions from graph theory, referring
  the readers to \cite{2010DiestelR-aa} for details. Let
$G = (V,E)$, $n=|V|$, $n_E = |E|$, be a simple, connected, undirected
graph. The graph Laplacian associated with $G$ is the matrix $L(G)$,
defined via the following bilinear form
\[
\langle L(G) u,v \rangle = \sum_{j=1}^n\sum_{k\in N(j)}
 (u_i-u_j)(v_i - v_j), \quad 
\forall u, \; v\in \R{n}. 
\]
Here, 
$\langle \cdot,\cdot \rangle$ is the standard Euclidean scalar product
on $\R{n}$, and 
$N(i)$ is the set of neighbors of $i\in V$; namely,
\(N(i) = \left\{j\in V \;\big|\; (i,j)\in E\right\}\).  Note that
\begin{eqnarray*}
&&[L(G)]_{ij}  = -1,\;\mbox{when}\; (i,j)\in E; \quad 
[L(G)]_{ij}  = 0, \;\mbox{when}\;(i,j)\notin E,\;\mbox{and}\;i\neq j;\\
&&[L(G)]_{ii}=-\sum_{j\in N(i)} [L(G)]_{ij}.
\end{eqnarray*}
In the following, we denote the eigenvalues of $L(G)$ by
$\lambda_1(G)\le \lambda_2(G),\ldots\le \lambda_n(G)$ and the
corresponding set of $\ell^2(\R{n})$-orthogonal eigenvectors by
$\varphi_1(G),\varphi_2(G),\ldots,\varphi_n(G)$.  We have that
$\lambda_1=0$ and
$\varphi_1(G) = (\underbrace{1,\ldots,1}_n)^T=:\bone_n$ for all
graphs, and that $\lambda_2(G)>0$ for all connected graphs $G$. The
eigenvalue $\lambda_2(G)$ is known as the \emph{algebraic
  connectivity} of $G$, denoted by $a(G)$.  The eigenvector
$\varphi_2(G)$ is known as the Fiedler vector, or \emph{characteristic
  valuation} of $G$ \cite{F1,F2,F3}. When $a(G)$ is a repeated
eigenvalue of multiplicity $k$, the Fiedler vectors
  lie in the $k$-dimensional eigenspace corresponding to $a(G)$,
  denoted here by $E_{\lambda_2}[L(G)]$.  When $G$ is clear from the
context, we will omit the argument ``$(G)$'' in
$\lambda_k(G)$, $\varphi_k(G)$ and simply write $\lambda_k$ and
$\varphi_k$.

Further, by $\bone_{W}$ we denote the indicator vector (function)
of a set. For a subset $W\subset V$ this function is defined as follows
\[
\bone_{W} \in \R{n},\quad 
[\bone_{W}]_k = \left\{
\begin{array}{ll} 1, & k\in W,\\
0, & k\notin W. 
\end{array}
\right. , \quad 
\bone_n=\bone_{\{1,\ldots,n\}}.
\]

We recall the usual definition of a median value $M(v)$ for $v\in
\R{n}$. If $w$ is a nondecreasing rearrangment of $v$, that is, there
exists a permutation $\pi$ such that $w=\pi v$ and $w_1\le
w_2\le\ldots\le w_n$, then
\[
M(v) = \left\{\begin{array}{ll}
w_{m}, & n=2m-1,\\
\frac{1}{2}(w_m+w_{m+1}), & n=2m.\\
\end{array}
\right.
\] 
The set of all cuts of $G$ is identified with the set of
decompositions of the set of vertices $V$ as a union of  $W\subset V$
and its complement $W^c\subset V$, isomorphic to the set of vectors $v\in \R{n}$ such that
$v_k\in \{-1,1\}$, $k=1:n$. For a vector
$v\in \R{n}$, its median cut is defined as corresponding to the vector
$C_M(v)$, where
\[
C_M(v) = \bone_{W}-\bone_{W^c}, \quad W=\{k\;\big|\; v_k > M(v)\}. 
\]

Adopting notation from \cite{spielman1996spectral}, we consider the graph partitioning problem, that is, partitioning $V$
into two disjoint sets $S$ and $S^c$, $S \cup S^c = V$, where
$|S| \approxeq |S^c|$, such that the edge cut (number of the edges with
one vertex in $S$ and one in $S^c$) is minimized. In the notation we
just introduced, this is equivalent to 
\begin{eqnarray*}
&&S= \arg\min\{F(W)\;\big|\; W\subsetneq V, \; W\neq \emptyset\},\\
&&F(W) = \frac{1}{2}\langle L(G) (\bone_{W}-\bone_{W^c}), \bone_{W}-\bone_{W^c}\rangle. 
\end{eqnarray*}
The relation $|S| \approxeq |S^c|$ quite ambiguous, and one of the most common approaches to resolve this is to
minimize over the sets $|W| = \left\lfloor \tfrac{n}{2} \right\rfloor$. In what follows, we will refer to this as a \emph{bisection}. An alternative is to minimize the cut ratio function, given by
\[
\operatorname{\phi}(W) = \frac{ F(W)} { \min \{ |W|, |W^c| \} },
\] 
which we refer to as a \emph{cut ratio partition}. Finding the minimizing bisection and cut ratio partition of a graph are both NP-complete problems \cite{garey1979computers}. In what follows, we consider solely the graph bisection problem. For more on the cut ratio partition, we refer readers to \cite{wei1989towards,wei1991ratio}.

Numerous techniques have been devised and aimed at approximating the
optimal solution in polynomial time and are currently used in
graph-partitioning software~\cite{Metis0,Metis1}. Some of the
techniques involve the use of a Fiedler vector of the graph
Laplacian
\cite{1990PothenA_SimonH_LiouK-aa,1988PowersD-aa,urschel2014cascadic}. Miroslav Fiedler,
\cite{F1,F2,F3}, proved many results about algebraic connectivity and
its associated eigenvector.  General results regarding the graph
Laplacian and its spectrum can be found in
\cite{M1,M4,M2,M3,Mohar91,1991MoharB-aa}. Following
Fiedler~\cite{F1,F2}, for $x \in \R{n}$ we define
\[
i_+(x):=\{i | x_i>0 \}, \quad i_-(x):=\{i | x_i<0 \}, \quad i_0(x)= \{i | x_i = 0 \}. 
\]
For a graph $G$, $i_0 (\varphi_2 - M(\varphi_2) \bone) = \emptyset$, the bisection via a Fiedler vector $\varphi_2$ is given by 
\[
S=i_+(\varphi_2- M( \varphi_2) \bone ), \quad S^c=i_-(\varphi_2 - M( \varphi_2 ) \bone ).
\]
Techniques of this form are called spectral bisections.  The
approximation can be seen by noting that minimizing the Rayleigh
quotient of the Laplacian over integer vectors is equivalent to
minimizing the edge cut of the corresponding
partition (see, e.g.~\cite{1997ChanT_CiarletP_SzetoW-aa,Lux}).
We note that the case when 
$i_0(\varphi_2- M( \varphi_2) \bone )$ is non-empty is more complicated and still
provides a bisection, depending on how the zero valuated vertices in
$(\varphi_2- M( \varphi_2) \bone)$ are distributed between $S$ and $S^c$. For
more details on such special cases, we refer to~\cite{urschel2014spectral}.

\section{Comparison: spectral bisection and optimal bisection}
In this section we show how far from optimal
a spectral bisection can be. We note that the area of
cut quality has been studied previously. Spielman and
Teng showed that spectral partitioning performs well for bounded
degree planar graphs \cite{spielman1996spectral}. Guattery and Miller
produced a class of graphs for which the error in cut quality of a spectral bisection is bounded below by a constant multiple of the order of the graph. In addition, they found estimates for how poorly spectral techniques can perform in the cut ratio partition
\cite{guattery1995performance,guattery1998quality}. By contrast, our
focus is solely on the bisection problem, and we aim to produce a lower bound of a constant multiple of the order squared.

We introduce a special graph Laplacian and calculate a spectral
decomposition for it. Given $n=4m$, $m>0$, we define
\[
L_* = \begin{pmatrix}
I & -I & 0 & 0\\
-I & (m+1)I & -\bone\bone^T & 0\\
0 & -\bone\bone^T & (m+1)I & -I\\
0 & 0  & -I & I
\end{pmatrix}.
\]
To avoid the proliferation of indices here and in what follows, we set
$\bone=\bone_m$.  To describe the eigenbasis of $L_*$, we use the
tensor product, defined as follows:
\[
\R{pr \times qs}\ni a \otimes b = 
\begin{pmatrix} 
a_{11} b & \ldots & a_{1q} b \\ 
\vdots & \vdots & \vdots \\  
a_{p1} b & \ldots & a_{pq} b 
\end{pmatrix}, \quad a\in \R{p\times q},\quad b\in \R{r\times s}. 
\] 
Simple calculations show that the
eigen-decomposition for $L_*$ is as follows:
\begin{enumerate}
\item One simple $0$ eigenvalue with eigenvector
  $\bone_n=(1,1,1,1)^T\otimes\bone$.
\item Two simple eigenvalues $\lambda_{\pm} = m+1 \pm \sqrt{m^2+1}$ with
  eigenvectors 
\[
\phi_{\pm} =
(1,(1-\lambda_{\pm}),(\lambda_{\pm}-1),-1)^T\otimes \bone.
\]
\item Two repeated eigenvalues 
$\mu_{\pm}= \frac{m}{2}+1\pm \sqrt{\left(\frac{m}{2}\right)^2+1}$,  
each of which has a eigenspace of dimension $(2m-2)$ with a basis
\begin{eqnarray*}
&&\psi^-_k = (1,(1-\mu_{-}),0,0)^T\otimes\xi_k, \quad k=1:m-1,\\
&&\psi^-_{k+m-1} = (0,0,(1-\mu_{-}),1)^T\otimes\xi_k, \quad k=1:m-1,\\
&&\psi^+_k =  (1,(1-\mu_{+}) ,0,0)^T\otimes\xi_k,  \quad k=1:m-1,\\
&&\psi^+_{k+m-1} =  (0,0,(1-\mu_{+}) ,1)^T\otimes\xi_k,  \quad k=1:m-1. 
\end{eqnarray*}
Here $\{\xi_k\}_{k=1}^{m-1}$ is a basis in $[\operatorname{span}\{\bone\}]^{\perp}$.  
\item One eigenvalue equal to $2$ with eigenvector 
$$\phi_2=(1,-1,-1,1)^T\otimes \bone.$$
\end{enumerate}

 We now consider a family of graphs for which the cut produced by
spectral bisection and the optimal bisection are ``far'' from each other. We
consider four connected undirected graphs, $G_k$, $k=1:4$, each with
$m$ vertices and corresponding graph Laplacians $L_k$, $k=1:4$. Let
$G_0=G_1 + G_2 + G_3 + G_4$ be the disjoint union of these graphs and
note that $G_0$ is a disconnected graph. Let $L_0$ be the graph
Laplacian associated with $G_0$. We now consider the graph $G$ (see
Figure~\ref{fig:graphs_example}) with Laplacian
\begin{equation}\label{eq:L}
L = L_0 + L_* =  
\begin{pmatrix}
L_1+I & -I & 0 & 0\\
-I & L_2+(m+1)I & -\bone\bone^T & 0\\
0 & -\bone\bone^T & L_3+(m+1)I & -I\\
0 & 0  & -I & L_4+I
\end{pmatrix}.
\end{equation}
\begin{wrapfigure}{r}{0.5\textwidth}
  \begin{center}
    \includegraphics[width=0.4\textwidth]{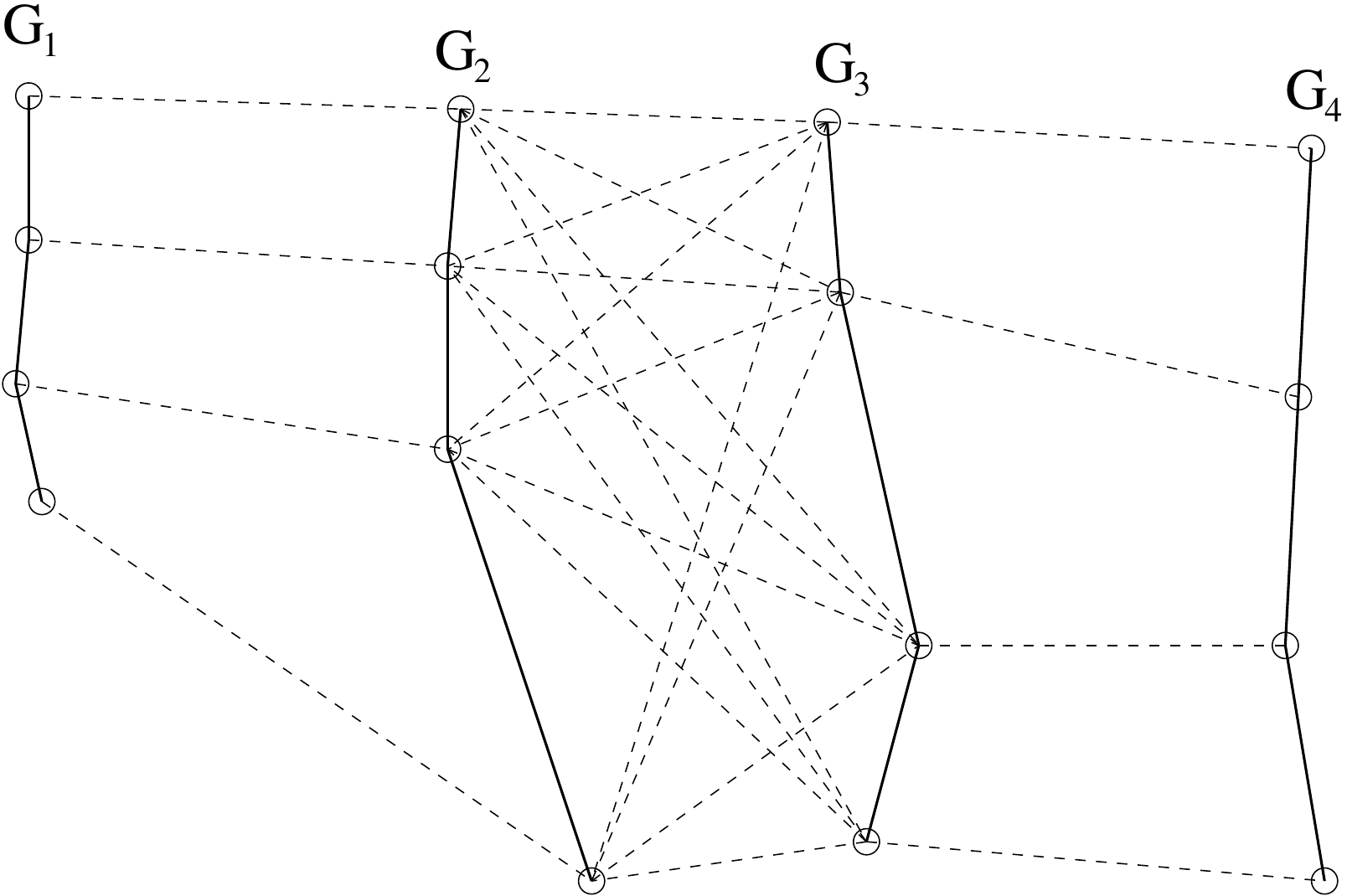}
  \end{center}
  \caption{Example of 4 graphs connected using $G_*$\label{fig:graphs_example}}
\end{wrapfigure}
The key concept in the construction of
$G$ is that the obvious best choice of $(1,-1,-1,1)^T\otimes \bone$
results in a bisection in which one of the resulting graphs is
disconnected. A main result of Fiedler shows that graphs generated by
$i_+(\varphi_2) \cup i_0(\varphi_2)$ and $i_-(\varphi_2) \cup i_0(\varphi_2)$ are necessarily
connected \cite{F2,urschel2014spectral}. This result implies that
$(1,-1,-1,1)^T\otimes \bone$ cannot be induced by a Fiedler vector.

We note that we have
$\operatorname{Ker}(L_0) = \mathcal{W}_{0}\oplus
\operatorname{span}\{\bone_n\}$, where
\begin{eqnarray*}
&& \mathcal{W}_0 =
   \operatorname{span}\{\phi_-,\phi_+,\phi_2\}.
\end{eqnarray*}
We then have the decomposition
\[
\R{n} = \operatorname{span}\{\bone_n\}\oplus \mathcal{W}_0 \oplus
\mathcal{W}_1,
\]
where
\begin{eqnarray*}
  && \mathcal{W}_1= 
     \operatorname{span}\{\psi^{-}_k,\psi^{+}_k\}_{k=1}^{2m-2} = \operatorname{Ker}(L_0)^{\perp}.
\end{eqnarray*}
Obviously, we have that
$\operatorname{Ker}(L)^{\perp}  = \mathcal{W}_0\oplus\mathcal{W}_1$
and that $\mathcal{W}_0$ and $\mathcal{W}_1$ are orthogonal. The following lemma relates the algebraic connectivity $a(G)$ and
the eigenvalues of $L_*$ under certain conditions. \\

\begin{lemma}  \label{algcon}
Let $a(G_k)$ be the algebraic connectivity of $G_k$, $k=1:4$. Assume
that $a(G_k) \ge \lambda_--\mu_-$ for all $k=1:4$. Then $a(G)=\lambda_-$, and the
corresponding Fiedler vector is $\phi_-$. \\
\end{lemma}
\begin{proof}
Because $\mathcal{W}_0$ and $\mathcal{W}_1$ are invariant subspaces, we necessarily have that the subspace of Fiedler vectors is contained in exactly one of the two spaces. We have that
\begin{eqnarray*}
\min_{\varphi \in \mathcal{W}_1} \frac{\langle L\varphi,\varphi
  \rangle}{\langle\varphi,\varphi\rangle} & \ge &
\min_{\varphi \in \mathcal{W}_1} \frac{\langle L_0\varphi,\varphi 
  \rangle}{\langle\varphi,\varphi\rangle} + \min_{\varphi \in \mathcal{W}_1}
\frac{\langle L_*\varphi,\varphi 
  \rangle}{\langle\varphi,\varphi\rangle} \\
& \ge &  \lbrack \min_{1\le k\le 4} a(G_k) \rbrack + \mu_- >  \lambda_--\mu_- + \mu_-=\lambda_-.
\end{eqnarray*}
Noting that
$$\min_{\varphi \in \mathcal{W}_0} \frac{\langle L\varphi,\varphi
  \rangle}{\langle\varphi,\varphi\rangle} = \min_{\varphi \in \mathcal{W}_0} \frac{\langle L_* \varphi,\varphi
  \rangle}{\langle\varphi,\varphi\rangle} = \lambda_-,$$
the desired result follows.
\end{proof}
\medskip

It is easy to estimate that for $m \ge 2$
$$
0<\lambda_--\mu_- \le \sqrt{ \frac{m^2}{4} + 1 } - \frac{m}{2} = \frac{1} {\sqrt{ \frac{m^2}{4} + 1 } + \frac{m}{2}   } < \frac{1}{m} .
$$ 
There are many examples of graphs for which $a(G) \ge \frac{1}{m}$,
including the path graph $P_m$, the cycle graph $C_m$, the complete
graph $K_m$, the bipartite complete graph $K_{p,q}$, and the
$n$-dimensional cube $Q_m$ \cite{de2007old}. We give the following lemma, which can be quickly verified and will be useful in proving the main theorem of the paper.

\begin{lemma} \label{addvertex} Let $L$ be the Laplacian of some simple, connected, undirected graph
  $G=(V,E)$,  with eigenvalues $\lambda_k$ and eigenvectors
  $\varphi_k$, $k=1,\ldots,n$. 
  Let $\widehat G=(\widehat{V},\widehat{E})$ be the graph obtained from
  $G$ by adding one vertex to $V$ and connecting it with every vertex
  in $G$, namely, $\widehat{V} = V\cup\{ n+1 \}$ and
  $\widehat{E} = E\cup\{(1,n+1),\ldots,(n,n+1)\}$.
  Then $L(\widehat G)$ has eigenvalues
  $\widehat \lambda_1=0$, $\widehat\lambda_k=\lambda_k +1$,
  $k=2,\ldots,n$,   $\widehat\lambda_{n+1} = n+1$ and 
the corresponding eigenvectors are
  $\widehat \varphi_1=\begin{pmatrix} \varphi_1 \\ 1 \end{pmatrix}$, 
$\widehat\varphi_k =  \begin{pmatrix}
    \varphi_k \\ 0 \end{pmatrix}$, $k=2,\ldots,n$, and 
$\widehat{\varphi}_{n+1} = \begin{pmatrix} \varphi_1 \\ -n \end{pmatrix}$. \\
\end{lemma}

\begin{proof}
By construction, the graph Laplacian corresponding to $\widehat G$  is
\[\widehat L = \begin{pmatrix} L+I & - \bone \\ -\bone^T &
    \langle\bone, \bone\rangle \end{pmatrix}.\]
 Verifying that
  $\widehat L \widehat\varphi_k=\widehat\lambda_k\widehat\varphi_k$
  for $k=1,\ldots,(n+1)$ is rather straightforward. This concludes the proof. 
\end{proof}

From Lemmas \ref{algcon} and \ref{addvertex}, the following theorem
regarding maximum spectral approximation error quickly follows. \\

\begin{theorem}\label{thm:1}
  Let $\cG_n$ be the set of simple, connected, undirected graphs of
  order $n$, $n>48$. Then the maximum spectral approximation error over
  all graphs $G \in \cG_n$ of the bisection has bounds
\begin{eqnarray*}
 \frac{n^2}{384} < \max_{G \in \cG_n} |   \min_{y \in E_{\lambda_2}[L] } F(i_+(y-M(y) \bone)) - \min_{A
  \subset V, |A|= \lfloor \frac{n}{2} \rfloor }
  F(A)| & < & \frac{n^2}{2}.
\end{eqnarray*}
\end{theorem}
\medskip

\begin{proof}
  Suppose we have $n = 4 m$ for some integer $m$. Let us choose
  $G \in \cG_n$ such that $L(G)$ corresponds to the Laplacian
  $L=L_0+L_*$, with $a(G_k) \ge \lambda_--\mu_-$, $k=1 :4$.
This gives us $a(G) = m+1 - \sqrt{m^2+1}$, with corresponding Fiedler vector $y = (1,(1-a(G)),(a(G)-1),-1)^T\otimes \bone$. We have
\begin{align*}
&F(i_+(y-M(y) \bone))  =  F(i_+(y)) = m^2 = \frac{n^2}{16}.
\end{align*}
We now consider the bisection induced by the vector $v = (1,-1,-1,1)^T\otimes \bone$. We have 
\begin{align*}
&F(i_+(v-M(v) \bone)) = F(i_+(v)) = 2 m = \frac{n}{2}.
\end{align*}

Now suppose $n \ne 4m$ for some integer $m$. Let $k = n \mod 4$, so
that $n = 4m +k$. Using the same graph $G \in \cG_{4m}$ as before,
let us now add $k$ vertices sequentially, with each addition adjacent
to every current vertex in the graph, as in Lemma \ref{addvertex}. By Lemma \ref{addvertex}, the Fiedler vector remains unchanged, with zeros in the entries of the added vertices.

In this case we have $i_0(y)$ non-empty. However, irrespective of how we choose to distribute these vertices, we still have similar bounds on the cut. The cut with respect to $y$ must necessarily increase. For the vector $\tilde v = \begin{pmatrix} v \\ 0 \end{pmatrix}$, we have the upper bound
\begin{align*}
&F(i_+(\tilde v) \cup U)  < \frac{(k+1)n}{2} \le 2n,
 \end{align*}
where $U \subset i_0(\tilde v)$. \\

Looking at both cases together, with $\tilde v = v$ for $k=0$, we have
\begin{eqnarray*}
 |F(i_+(y) \cup U_1) - \min_{A \subset V, |A|= \lfloor \frac{n}{2} \rfloor }  F(A) | & \ge & |F(i_+(y) \cup U_1) -  F(i_+(\tilde v) \cup \tilde U_1) | \\
& \ge &  \big| \frac{n^2}{16} - 2n \big| >   \big| \frac{n^2}{16} - \frac{n^2}{24} \big| = \frac{n^2}{384}
\end{eqnarray*}
for any $U_1 \subset i_0(y)$, $\tilde U_1 \subset i_0( \tilde
v)$.
That completes the lower bound. The upper bound results from
considering an upper bound on the size of the graph.
\end{proof}
\medskip

\section*{Acknowledgments}

The authors would like to thank Louisa Thomas for improving the style of presentation.


\begin{thebibliography}{10}
\providecommand{\url}[1]{\normalfont{#1}}
\providecommand{\urlprefix}{Available from: }

\bibitem{garey1979computers}
Garey~MR, Johnson~DS. Computers and intractibility: a guide to the theory of
  np-completeness. WH Freeman and Company, New York. 1979;\hspace{0pt}18:41.

\bibitem{Metis0}
Karypis~G, Kumar~V. A parallel algorithm for multilevel graph partitioning and
  sparse matrix ordering. Journal of Parallel and Distributed Computing.
  1998;\hspace{0pt}48:71--85;
  \urlprefix\url{http://glaros.dtc.umn.edu/gkhome/node/106}.

\bibitem{Metis1}
Karypis~G, Kumar~V. Parallel multilevel {$k$}-way partitioning scheme for
  irregular graphs. SIAM Rev. 1999;\hspace{0pt}41(2):278--300 (electronic);
  \urlprefix\url{http://dx.doi.org/10.1137/S0036144598334138}.

\bibitem{1990PothenA_SimonH_LiouK-aa}
Pothen~A, Simon~HD, Liou~KP. Partitioning sparse matrices with eigenvectors of
  graphs. SIAM J Matrix Anal Appl. 1990;\hspace{0pt}11(3):430--452; sparse
  matrices (Gleneden Beach, OR, 1989);
  \urlprefix\url{http://dx.doi.org/10.1137/0611030}.

\bibitem{1988PowersD-aa}
Powers~DL. Graph partitioning by eigenvectors. Linear Algebra Appl.
  1988;\hspace{0pt}101:121--133;
  \urlprefix\url{http://dx.doi.org/10.1016/0024-3795(88)90147-4}.

\bibitem{urschel2014cascadic}
Urschel~JC, Hu~X, Xu~J, Zikatanov~LT. A cascadic multigrid algorithm for
  computing the fiedler vector of graph laplacians. arXiv preprint
  arXiv:14120565. 2014;\hspace{0pt}.

\bibitem{spielman1996spectral}
Spielman~DA, Teng~SH. Spectral partitioning works: Planar graphs and finite
  element meshes. In: Foundations of Computer Science, 1996. Proceedings., 37th
  Annual Symposium on. IEEE; 1996. p. 96--105.

\bibitem{2010DiestelR-aa}
Diestel~R. Graph theory. 4th ed; Vol. 173 of Graduate Texts in Mathematics.
  Heidelberg: Springer; 2010.

\bibitem{F1}
Fiedler~M. Algebraic connectivity of graphs. Czechoslovak Math J.
  1973;\hspace{0pt}23(98):298--305.

\bibitem{F2}
Fiedler~M. A property of eigenvectors of nonnegative symmetric matrices and its
  application to graph theory. Czechoslovak Math J.
  1975;\hspace{0pt}25(100)(4):619--633.

\bibitem{F3}
Fiedler~M. Laplacian of graphs and algebraic connectivity. In: Combinatorics
  and graph theory ({W}arsaw, 1987). Vol.~25 of Banach Center Publ.; Warsaw:
  PWN; 1989. p. 57--70.

\bibitem{wei1989towards}
Wei~YC, Cheng~CK. Towards efficient hierarchical designs by ratio cut
  partitioning. In: Computer-Aided Design, 1989. ICCAD-89. Digest of Technical
  Papers., 1989 IEEE International Conference on. IEEE; 1989. p. 298--301.

\bibitem{wei1991ratio}
Wei~YC, Cheng~CK. Ratio cut partitioning for hierarchical designs.
  Computer-Aided Design of Integrated Circuits and Systems, IEEE Transactions
  on. 1991;\hspace{0pt}10(7):911--921.

\bibitem{M1}
Merris~R. Laplacian matrices of graphs: a survey. Linear Algebra Appl.
  1994;\hspace{0pt}197/198:143--176; second Conference of the International
  Linear Algebra Society (ILAS) (Lisbon, 1992);
  \urlprefix\url{http://dx.doi.org/10.1016/0024-3795(94)90486-3}.

\bibitem{M4}
Merris~R. A survey of graph {L}aplacians. Linear and Multilinear Algebra.
  1995;\hspace{0pt}39(1-2):19--31;
  \urlprefix\url{http://dx.doi.org/10.1080/03081089508818377}.

\bibitem{M2}
Merris~R. Laplacian graph eigenvectors. Linear Algebra Appl.
  1998;\hspace{0pt}278(1-3):221--236;
  \urlprefix\url{http://dx.doi.org/10.1016/S0024-3795(97)10080-5}.

\bibitem{M3}
Merris~R. A note on {L}aplacian graph eigenvalues. Linear Algebra Appl.
  1998;\hspace{0pt}285(1-3):33--35;
  \urlprefix\url{http://dx.doi.org/10.1016/S0024-3795(98)10148-9}.

\bibitem{Mohar91}
Mohar~B. The laplacian spectrum of graphs. In: Graph Theory, Combinatorics, and
  Applications. Wiley; 1991. p. 871--898.

\bibitem{1991MoharB-aa}
Mohar~B. The {L}aplacian spectrum of graphs. In: Graph theory, combinatorics,
  and applications. {V}ol.\ 2 ({K}alamazoo, {MI}, 1988). Wiley-Intersci. Publ.;
  New York: Wiley; 1991. p. 871--898.

\bibitem{1997ChanT_CiarletP_SzetoW-aa}
Chan~TF, Ciarlet~P~Jr, Szeto~WK. On the optimality of the median cut spectral
  bisection graph partitioning method. SIAM J Sci Comput.
  1997;\hspace{0pt}18(3):943--948;
  \urlprefix\url{http://dx.doi.org/10.1137/S1064827594262649}.

\bibitem{Lux}
Luxburg~U. A tutorial on spectral clustering. Statistics and Computing. 2007
  Dec;\hspace{0pt}17(4):395--416;
  \urlprefix\url{http://dx.doi.org/10.1007/s11222-007-9033-z}.

\bibitem{urschel2014spectral}
Urschel~JC, Zikatanov~LT. Spectral bisection of graphs and connectedness.
  Linear Algebra and its Applications. 2014;\hspace{0pt}449:1--16.

\bibitem{guattery1995performance}
Guattery~S, Miller~GL. On the performance of spectral graph partitioning
  methods. In: Proceedings of the sixth annual ACM-SIAM symposium on Discrete
  algorithms. Society for Industrial and Applied Mathematics; 1995. p.
  233--242.

\bibitem{guattery1998quality}
Guattery~S, Miller~GL. On the quality of spectral separators. SIAM Journal on
  Matrix Analysis and Applications. 1998;\hspace{0pt}19(3):701--719.

\bibitem{de2007old}
de~Abreu~NMM. Old and new results on algebraic connectivity of graphs. Linear
  algebra and its applications. 2007;\hspace{0pt}423(1):53--73.

\end{thebibliography}
\end{document}